\documentclass[a4paper, reqno, 14pt]{amsart}

\usepackage[usenames,dvipsnames]{color}
\usepackage{amsthm,amsfonts,amssymb,amsmath,amsxtra, appendix}
\usepackage[all]{xy}
\SelectTips{cm}{}
\usepackage{xr-hyper}
\usepackage{verbatim}
\usepackage[margin=1.25in]{geometry}
\usepackage{mathrsfs}
\usepackage{diagbox}
\usepackage{enumerate}
\usepackage[shortlabels]{enumitem}
\usepackage{tikz}
\usepackage{tkz-euclide}
\usepackage{setspace} 
\usepackage{hyperref}
\RequirePackage{xspace}
\RequirePackage{etoolbox}
\RequirePackage{varwidth}
\RequirePackage{enumitem}
\RequirePackage{tensor}
\RequirePackage{mathtools}
\RequirePackage{longtable}
\RequirePackage{multirow}

\usepackage{tikz}
\usepackage{tikz-cd}
\usepackage{csquotes}

\usepackage[url=false,doi=false,isbn=false,sorting=nyt,giveninits=true,maxbibnames=4]{biblatex}
\def\ge{\geqslant}
\def\le{\leqslant}
\def\a{\alpha}
\def\b{\beta}
\def\g{\gamma}
\def\G{\Gamma}

\def\D{\Delta}
\def\L{\Lambda}

\def\s{\sigma}
\def\t{\tau}
\def\th{\theta}

\def\l{\lambda}

\def\af{\mathrm{af}}
\def\aff{\mathrm{aff}}

\def\QBG{\mathrm{QBG}}

\def\<{\langle}
\def\>{\rangle}

\newcommand{\bG}{\mathbf G}

\newcommand{{\BG}}{\ensuremath{\mathbb {G}}\xspace}

\newcommand{{\BK}}{\ensuremath{\mathbb {K}}\xspace}

\newcommand{\BR}{\ensuremath{\mathbb {R}}\xspace}
\newcommand{\BS}{\ensuremath{\mathbb {S}}\xspace}

\newcommand{\BZ}{\ensuremath{\mathbb {Z}}\xspace}

\newcommand{\CC}{\ensuremath{\mathcal {C}}\xspace}

\newcommand{\CR}{\ensuremath{\mathcal {R}}\xspace}

\newcommand{\CU}{\ensuremath{\mathcal {U}}\xspace}

\newcommand{\Ad}{{\mathrm{Ad}}}

\DeclareMathOperator{\Adm}{Adm}

\DeclareMathOperator{\Gal}{Gal}

\newcommand{\id}{\ensuremath{\mathrm{id}}\xspace}

\newcommand{\wt}{\mathrm{wt}}

\def\tW{\tilde W}

\DeclareMathOperator{\supp}{supp}

\newtheorem{theorem}{Theorem}
\newtheorem{proposition}[theorem]{Proposition}
\newtheorem{lemma}[theorem]{Lemma}
\newtheorem {conjecture}[theorem]{Conjecture}
\newtheorem{corollary}[theorem]{Corollary}

\theoremstyle{definition}
\newtheorem{definition}[theorem]{Definition}

\newtheorem*{example*}{Example}

\newtheorem{remark}[theorem]{Remark}

\newtheorem*{function*}{Function}

\numberwithin{equation}{section}
\numberwithin{theorem}{section}

\setitemize[0]{leftmargin=*,itemsep=\the\smallskipamount}
\setenumerate[0]{leftmargin=*,itemsep=\the\smallskipamount}

\renewcommand{\to}{%
   \ifbool{@display}{\longrightarrow}{\rightarrow}%
   }
\let\shortmapsto\mapsto
\renewcommand{\mapsto}{%
   \ifbool{@display}{\longmapsto}{\shortmapsto}%
   }
\newlength{\olen}
\newlength{\ulen}
\newlength{\xlen}
\newcommand{\xra}[2][]{%
   \ifbool{@display}%
      {\settowidth{\olen}{$\overset{#2}{\longrightarrow}$}%
       \settowidth{\ulen}{$\underset{#1}{\longrightarrow}$}%
       \settowidth{\xlen}{$\xrightarrow[#1]{#2}$}%
       \ifdimgreater{\olen}{\xlen}%
          {\underset{#1}{\overset{#2}{\longrightarrow}}}%
          {\ifdimgreater{\ulen}{\xlen}%
             {\underset{#1}{\overset{#2}{\longrightarrow}}}
             {\xrightarrow[#1]{#2}}}}%
      {\xrightarrow[#1]{#2}}
   }
\makeatother
\newcommand{\xyra}[2][]{%
   \settowidth{\xlen}{$\xrightarrow[#1]{#2}$}%
   \ifbool{@display}%
      {\settowidth{\olen}{$\overset{#2}{\longrightarrow}$}%
       \settowidth{\ulen}{$\underset{#1}{\longrightarrow}$}%
       \ifdimgreater{\olen}{\xlen}%
          {\mathrel{\xymatrix@M=.12ex@C=3.2ex{\ar[r]^-{#2}_-{#1} &}}}%
          {\ifdimgreater{\ulen}{\xlen}%
             {\mathrel{\xymatrix@M=.12ex@C=3.2ex{\ar[r]^-{#2}_-{#1} &}}}
             {\mathrel{\xymatrix@M=.12ex@C=\the\xlen{\ar[r]^-{#2}_-{#1} &}}}}}%
      {\mathrel{\xymatrix@M=.12ex@C=\the\xlen{\ar[r]^-{#2}_-{#1} &}}}%
   }
\makeatletter
\newcommand{\xla}[2][]{%
   \ifbool{@display}%
      {\settowidth{\olen}{$\overset{#2}{\longleftarrow}$}%
       \settowidth{\ulen}{$\underset{#1}{\longleftarrow}$}%
       \settowidth{\xlen}{$\xleftarrow[#1]{#2}$}%
       \ifdimgreater{\olen}{\xlen}%
          {\underset{#1}{\overset{#2}{\longleftarrow}}}%
          {\ifdimgreater{\ulen}{\xlen}%
             {\underset{#1}{\overset{#2}{\longleftarrow}}}
             {\xleftarrow[#1]{#2}}}}%
      {\xleftarrow[#1]{#2}}
   }
\newcommand{\isoarrow}{%
   \ifbool{@display}{\overset{\sim}{\longrightarrow}}{\xrightarrow\sim}%
   }

\begin{document}

\title[]{Dual Shellability of Admissible Set and Cohen-Macaulayness of Local Models}

\author[Xuhua He]{Xuhua He}
\address{Department of Mathematics and New Cornerstone Science Laboratory, The University of Hong Kong, Pokfulam, Hong Kong, Hong Kong SAR, China}
\email{xuhuahe@hku.hk}

\author[Qingchao Yu]{Qingchao Yu}
\address{Institute for Advanced Study, Shenzhen University, Nanshan District, Shenzhen, Guangdong, China}
\email{qingchao\_yu@outlook.com}

\thanks{}



\begin{abstract}
We prove G\"ortz’s combinatorial conjecture \cite{Go01} on dual shellability of admissible sets in Iwahori-Weyl groups, proving that the augmented admissible set $\widehat{\mathrm{Adm}}(\mu)$ is dual shellable for any dominant coweight $\mu$. This provides a uniform, elementary approach to establishing Cohen–Macaulayness of the special fibers of the local models with Iwahori level structure for all reductive groups—including residue characteristic $2$ and non-reduced root systems—circumventing geometric methods. Local models, which encode singularities of Shimura varieties and moduli of shtukas, have seen extensive study since their introduction by Rapoport–Zink, with Cohen–Macaulayness remaining a central open problem. While previous work relied on case-specific geometric analyses (e.g., Frobenius splittings \cite{HR23} or compactifications \cite{He13}), our combinatorial proof yields an explicit labeling that constructs the special fiber by sequentially adding irreducible components while preserving Cohen–Macaulayness at each step, a new result even for split groups.
\end{abstract}

\maketitle


\section*{Introduction}
\subsection{Local models and their singularities} 
Local models are projective schemes over a DVR, encoding the \'etale-local structure of Shimura varieties \cite{KP18} and the moduli stacks of shtukas \cite{AH19} with parahoric level structures. Early work by Deligne–Pappas \cite{DP94}, Chai–Norman \cite{CN90, CN92}, and de Jong \cite{dJ93} linked Shimura variety singularities to these models, later formalized by Rapoport–Zink \cite{RZ96}. General constructions emerged via Zhu \cite{Zhu14} and Pappas–Zhu \cite{PZ13} for tamely ramified groups, with extensions by \cite{Le16, Lo19+, FHLR, Ri16}. 

De Jong \cite{dJ93} pioneered the study of local model singularities, later expanded in \cite{Fa01, Go01, PRS, Ki09}. For large residue characteristics, \cite{PZ13} proved flatness with normal, Cohen–Macaulay components; normality follows from the generic fiber. For small residue characteristics, it requires seminormalization \cite{HLR}.

It was pointed out in \cite[\S 2.1]{PRS} that ``the question of Cohen-Macaulayness and normality of local models is a major open problem in the field''. When the special fiber is irreducible, flatness ensures the model is Cohen–Macaulay. This applies to special parahoric levels and a few other cases classified in \cite{HY24}. In general, the special fiber is the union of multiple Schubert varieties in the partial affine flag variety, indexed by the admissible set in the affine Weyl group. Görtz \cite{Go01} proposes to attack the question of Cohen-Macaulayness through a purely combinatorial problem, more precisely, the dual shellability of these admissible sets. He verified this combinatorial condition for $GL_n$ with $n \le 6$ using computer-assisted calculations. However, progress since has relied on geometric methods.

In \cite{He13}, the author related the local model associated with unramified groups and minuscule coweights to the De Concini-Procesi compactification of reductive groups, and deduced the Cohen-Macaulayness of the local models in these cases from the Cohen-Macaulayness of the group compactification. In \cite{HR23}, Haines and Richarz established the Cohen-Macaulayness of the local models in most cases. The equal-characteristic case was established via the powerful method of Frobenius splittings on the whole local model, and the mixed characteristic case was deduced from the equal characteristic case by the Coherence conjecture established by X. Zhu in \cite{Zhu14}. In small residue characteristics, the geometric structure of the local models is more complicated. Progress in these cases appears in \cite{FHLR, AGLR} and ongoing work by Gleason–Lourenço and Cass–Lourenço. Currently, Cohen–Macaulayness remains unsolved for residue characteristic $2$ and non-reduced root systems.

\subsection{Dual shellability of admissible sets} This paper resolves G\"ortz’s combinatorial conjecture, establishing dual shellability of admissible sets, and deduces the Cohen–Macaulayness for local models with Iwahori level structure for all reductive groups. This elementary approach circumvents geometric methods and handles bad residue characteristics uniformly. 

Let $\CR$ be a based root datum and $\tW$ be the associated Iwahori-Weyl group. In \cite{KR}, Kottwitz and Rapoport introduced the admissible set, an important subset of $\tW$. For any dominant coweight $\mu$, the {\it admissible set} $\Adm(\mu)$ is defined to be $$\Adm(\mu)=\{w \in \tW \mid w \le t^{\mu'} \text{ for some coweight } \mu' \text{ conjugate to } \mu\}.$$

The admissible set plays an important role in arithmetic geometry and representation theory. Geometrically, reductions of Shimura varieties with parahoric level structure decompose into disjoint Kottwitz–Rapoport strata indexed by admissible sets, while their local models similarly stratify into affine Schubert varieties parametrized by these sets. Algebraically, admissible sets encode the support of Bernstein functions in the center of affine Hecke algebras. Combinatorially, they capture intricate data about alcoves in the affine Bruhat–Tits building, forming a rich yet highly structured object. 

Shellability, a foundational concept in combinatorial geometry and poset topology, originated in the study of polyhedral complexes. A simplicial complex is said to be shellable if its maximal faces can be ordered in such a way that each face intersects the union of the previous faces in a pure and codimension-one manner. This property has significant implications for the algebraic properties of the complex, particularly the regularity of the complex and the Cohen-Macaulayness. Applications extend to the theory of total positivity, where shellability underpins the regularity of the totally positive flag varieties, as demonstrated in \cite{GKL,BH}.

EL-shellability generalizes the notion of shellability for posets via edge-labeling conditions: a poset is EL-shellable if its intervals admit edge labelings with unique increasing and lexicographically minimal chains. This allows for a more flexible and powerful framework for studying the combinatorial and topological properties of posets. One significant application of (dual) EL shellability is in the study of the union of Schubert varieties, as in \cite{Go01}.

\subsection{Main results}
The maximal elements of $\Adm(\mu)$ are $t^{\mu'}$ for $\mu' \in W_0(\mu)$, the $W_0$-orbit of $\mu$. To study the shellability, one considers the augmented admissible set $\widehat{\Adm (\mu)}=\Adm (\mu) \sqcup \{\hat 1\}$, where $\hat{1}$ is the added unique maximal element. Our main theorem is:

\begin{theorem}[Theorem \ref{thm:main}]
    For any dominant coweight $\mu$, the augmented admissible set $\widehat{\Adm (\mu)}$ is dual shellable.
\end{theorem}

As a consequence, we deduce the Cohen-Macaulayness for the special fibers of local models with Iwahori level structure for all reductive groups. For minuscule $\mu$, this implies the Cohen-Macaulayness of the local models themselves, as constructed in \cite{AGLR} for wildly ramified groups. For non-minuscule $\mu$, our result applies to the equicharacteristic analogues and to the known constructions of local models (when the residue characteristic is not $2$ or the group is not ramified unitary group). This elementary approach circumvents geometric methods and handles bad residue characteristics.




We emphasize that dual shellability implies Cohen–Macaulayness, but the converse implication is not always true (see \cite[p. 183]{Bj}). Beyond applications to the local models, we anticipate applications of shellability to total positivity of global Schubert varieties.

Moreover, since our construction yields an explicit labeling, we demonstrate how the irreducible components of the special fiber can be added sequentially, one-by-one, while preserving Cohen-Macaulayness at each step. This result is new even for split groups.

We expect that ideas in our construction can be applied to other interesting problems. Inside $\Adm(\mu)$, there is an interesting subset ${}^K \text{Cox}(\mu)$ consisting of certain Coxeter elements. The poset ${}^K \text{Cox}(\mu)$ is the index set of the EKOR-strata for basic loci of Shimura varieties associated with Shimura data of Coxeter type. We propose the following conjecture.
\begin{conjecture}
The poset $\widehat{{}^K \text{Cox}(\mu)}={}^K \text{Cox}(\mu) \sqcup \{\hat 1\}$ is dual EL-shellable.
\end{conjecture}

\subsection{Strategy of the proof}

The construction of our dual EL-labeling on $\widehat{\mathrm{Adm}}(\mu)$ proceeds via a two-step approach, carefully designed to ensure compatibility between different parts of the poset:
\begin{itemize}
    \item For edges within $\mathrm{Adm}(\mu)$, we employ a specific choice of reflection order (Lemma \ref{lem:ref_ord}) to define a labeling $\eta$, which restricts to a dual EL-labeling on Bruhat intervals $[\hat{0}, t^{a(\mu)}]$ by Dyer's work \cite[Proposition 4.3]{Dyer93}.

    \item For the top edges $\hat{1} \gtrdot t^{a(\mu)}$, we introduce a refined labeling based on the Bruhat order of the affine Weyl group element $a$.
\end{itemize}

Next, we give an explicit description of the top two layers of $\mathrm{Adm}(\mu)$ (Proposition \ref{prop:top-two}) via the quantum Bruhat graph introduced by Fomin, Gelfand, Postnikov \cite{FGP97}, and the acute presentation motivated by the work of Haines and Ngo \cite{HN02} and the work of Schremmer \cite{Sch24}.

The heart of the proof of dual EL-shellability lies in studying maximal chains within intervals $[w, \hat{1}]$. Using the recent developments on the quantum Bruhat graphs \cite{Sch24} and \cite{HY24}, we identify a unique minimal element $a_{\min}$ satisfying $w \leq t^{a_{\min}(\mu)}$. This element serves as the foundation for constructing the distinguished label-increasing chain in $[w, \hat{1}]$ that is lexicographically minimal (see \S\ref{sec:4}).

{\bf Acknowledgements: } XH is partially supported by the New Cornerstone Science Foundation through the New Cornerstone Investigator Program and the Xplorer Prize, and by Hong Kong RGC grant 14300122. QY is partially supported by the National Natural Science Foundation of China (grant no. 12501018). We thank Thomas Haines and Felix Schremmer for valuable suggestions and comments. 

\section{Preliminary}
\subsection{Shellability}\label{sec:1.1}
Let $(P,\le)$ be a finite poset (partially ordered set). Let $w,w' \in P$. Define $[w,w'] = \{z\in P\mid w\le z\le w'\}$ and $[w,w')=\{z\in P\mid w\le z < w'\}$. If $w' > w$ and there is no $z\in P$ such that $w'>z>w$, we write $w' \gtrdot w$ (or $w\lessdot w'$) and call it a covering in $P$ or an edge of $P$. Let $\mathscr{E}(P)$ be the set of all edges of $P$.

A \textit{downward chain} from $w'$ to $w$ in $P$ of \textit{length} $r$ is a sequence of elements $w'>  w_{r-1} > \ldots > w_1 > w$ in $P$. We say a downward chain is \textit{maximal}, if it is not a subchain of any other downward chain. 

In this paper, all chains are downward chains.

We say that a finite poset $P$ is \textit{pure}, if all maximal chains in $P$ have the same length. We call this common length the \emph{length} of $P$. If moreover, $P$ contains a unique minimal element $\hat{0}$ and a unique maximal element $\hat{1}$, we say that $P$ is a \textit{graded} poset.

Let $P$ be a graded poset. An edge labeling of $P$ is a map $\eta: \mathscr{E}(P) \to (\L,\preceq)$, where $(\L,\preceq)$ is a totally ordered poset. Let $\eta$ be an edge labeling of $P$. A maximal chain $\hat{1} = w_r \gtrdot w_{r-1} \gtrdot \cdots \gtrdot w_1=  \hat{0}$ is said to be \emph{label-increasing}, if $\eta(x_{i+1 } \gtrdot x_{i }) \preceq \eta(x_{i } \gtrdot x_{i-1})$ for $i =  2,3,\ldots, r-1$. Similarly define the notion of \emph{label-decreasing}.

Let $c : \hat{1} = w_r \gtrdot w_{r-1} \gtrdot \cdots \gtrdot w_1=  \hat{0}$ and $c' : \hat{1} = w_r' \gtrdot w_{r-1}' \gtrdot \cdots \gtrdot w_1'=  \hat{0}$ be two maximal chains in $P$. If the label sequence $(\eta(w_r \gtrdot w_{r-1}),\ldots, \eta(w_{2} \gtrdot w_1))$ lexicographically precedes $(\eta(w_r' \gtrdot w_{r-1}'),\ldots, \eta(w_{2}' \gtrdot w_1'))$, we say that $c$ lexicographically precedes $c'$ and write $c \prec_{\text{lex}} c'$.


\begin{definition}\label{def:EL}
    An edge labeling $\eta: \mathscr{E}(P) \to (\L,\preceq)$ is said to be a \emph{dual EL-labeling}, if for any interval $[w, w']$ of $P$,
    \begin{enumerate}
        \item there exists a unique maximal chain in $[w,w']$ which is label-increasing;
        \item the chain in (1) is lexicographically minimal among all maximal chains in $[w,w']$. 
    \end{enumerate}
If $P$ admits a dual EL-labeling, we say that $P$ is \textit{dual EL-shellable}.
\end{definition}


\begin{remark}
If $P$ is dual EL-shellable, we say that the dual poset $P^{\vee}$ is EL-shellable. We only study dual EL-shellability in this paper. 

\end{remark}

\subsection{Related properties on posets}\label{sec:1.2}
In this subsection, we recall some properties
related to shellability.

A \emph{coatom} of the graded poset $P$ is an element covered by $\hat{1}$, i.e., an element $w\in P$ with $\hat{1}\gtrdot w$.
\begin{definition}[{\cite[Definition 3.1]{BW83}}]\label{def:coatom}
    We say that $P$ \emph{admits a recursive coatom ordering} if the length of $P$ is $1$ or if the length of $P$ is greater than $1$ and there is an ordering  $x_1, x_2, \ldots, x_t $ of the coatoms of $P$ which satisfies:
\begin{enumerate}
\item For any $ j = 1, 2, \ldots, t $, $ [\hat{0},x_j] $ admits a recursive coatom ordering in which the coatoms of $[\hat{0},x_j]$ that come first in the ordering are those that are covered by some $ x_i $ where $i < j$.

\item If $w <x_i,x_j$ with $i<j$, then there exists $ k < j$ and an element $w' $ such that $w\le w'$ and $w' $ is covered by $x_k$ and $x_j$.

\end{enumerate}
\end{definition}

The following result is due to Bj\"orner-Wachs.\footnote{\cite[Theorem 3.2]{BW83} proves that $P$ is dual CL-shellable if and only if $P$ admits a recursive coatom ordering. And EL-shellable $\implies$ CL-shellable is obvious. We do not need the notion of (dual) CL-shellable in this paper.}
\begin{theorem}[{\cite[Theorem 3.2]{BW83}}]\label{thm:BW}
If $P$ is dual EL-shellable, then $P$ admits a recursive coatom ordering.
\end{theorem}

Following G\"ortz, we have the notion of \emph{$N$-Cohen-Macaulay poset}.

\begin{definition}[{\cite[Definition 4.23]{Go01}}]\label{def:CM}
Let $Q$ be a pure poset with a unique minimal element $\hat{0}$. For any $z\in Q$, the length of $z$ is the length of any maximal chain in $[\hat{0},z]$. We define the notion of $N$-Cohen-Macaulayness recursively as follows.
\begin{enumerate}
    \item If $Q$ has a unique maximal element $ x $, then $ Q $ is $\ell(x)$-Cohen-Macaulay, where $\ell(x)$ is the length of $x$. 
    \item Now denote by $ x_1, \ldots, x_k $ the maximal elements of $ Q $, and suppose that $ k \geq 2 $. The set $Q$ is called $N$-Cohen-Macaulay, if all the $ x_i $ have length $ N $, and if, after possibly changing the order of the \( x_i \)'s, we have that for all \( j = 2, \ldots, k \), the set
    $$
    [\hat{0},x_j)\cap \left(\cup_{i<j}[\hat{0},x_i)\right)
    $$
    is $ (N - 1) $-Cohen-Macaulay.
\end{enumerate}
\end{definition}

\begin{lemma}\label{lem:coatom}
If a graded poset $P$ of length $N+1$ admits a recursive coatom ordering, then $P-\{\hat{1}\}$ is $N$-Cohen-Macaulay.
\end{lemma}
\begin{proof}
We prove by induction on $N$. If $N=1$, the statement is trivial. Suppose $N>1$. Let $x_1,x_2,\ldots,x_t$ be the ordering of the coatoms of $P$ as in Definition \ref{def:coatom}. Then for each $j$, $[\hat{0},x_j]$ admits a recursive coatom ordering $y_1,y_2,\ldots,y_m,y_{m+1},\ldots, y_{\ell}$ such that the maximal elements of $ [\hat{0},x_j)\cap \left(\cup_{i<j}[\hat{0},x_i)\right)$ are $y_1,y_2,\ldots,y_m$. The recursive coatom ordering of $[\hat{0},x_j]$ restricts to a recursive coatom ordering of $\{x_j\}\sqcup \left(\cup_{k\le m} [\hat{0},y_{k}] \right)$. By induction hypothesis, the set $ \cup_{k\le m} [\hat{0},y_{k}] = [\hat{0},x_j)\cap \left(\cup_{i<j}[\hat{0},x_i)\right)$ is $(N-1)$-Cohen-Macaulay.
\end{proof}

In summary, we have
\begin{align*}
     \text{dual EL-shellable}\implies  \text{recursive coatom ordering}\implies N\text{-Cohen-Macaulay}.
\end{align*}


\subsection{The Affine Weyl Group}\label{sec:1.3}

Let $\CR = (\Phi, X^*,\Phi^{\vee},X_* ,\D_0)$ be a based root datum. The set of simple roots is $\{\a_i \mid i\in \D_0\}$ and the set of simple coroots is $\{\a_i^{\vee} \mid i\in \D_0\}$. Let $W_0$ be its Weyl group. Let $\Phi^+$ and $\Phi^-$ be the sets of positive and negative roots respectively. Set $\rho = \frac{1}{2}\sum_{\a\in \Phi^+} \a$. Let $\<-,-\>$ be the pairing between $X_*$ and $X^*$.

Let $\tW = X_*\rtimes W_0 = \{ t^{\l}z \mid \l \in X_*, z\in W_0\}$ be the extended affine Weyl group. The set of affine roots is defined as $\Phi_{\aff} = \Phi \times \BZ$. By convention, we choose the set of positive affine roots as $\Phi_{\aff}^+ = (\Phi^+ \times \BZ_{\ge0} )\sqcup (\Phi^- \times \BZ_{\ge1})$. For any $w = t^{\l}z \in   \tW$, $\tilde{\a} = (\a,k)\in \Phi\times\BZ$ and $v \in V$, the action of $w$ on $\tilde{\a}$ is given by $w(\a,k) = (z(\a), k-\<\l,z(\a)\>)$, the action of $w$ on $v$ is given by $w(v) = \l + z(v)$, the affine reflection corresponding to $\tilde\a$ is $s_{\tilde{\a}} = s_{\a} t^{k\a^{\vee}} \in \tW$, and the corresponding hyperplane is $H_{\tilde\a} = \{v\in V \mid \<v,\a\> = - k\}$. For $w\in \tW$ and $\tilde\a\in \Phi^+_{\aff}$, $ws_{\tilde\a}>w$ if and only if $w(\tilde{\a})\in \Phi^+_{\aff}$.

Let $\D_{\aff}\supseteq\D_0$ be the index set of affine simple roots. For any $i \in \D_0$, the corresponding affine simple root is $\tilde\a_i = (\a_i,0)$. For any $i \in \D_{\aff} - \D_0$, the corresponding affine simple root is of the form $\tilde{\a}_i = (-\th, 1)$, where $\th$ is the highest root in the corresponding irreducible component of $\Phi$. Let $\BS_0 = \{s_i \mid i\in \D_0\}$ and $\tilde{\BS} = \{s_i \mid i\in \D_{\aff}\}$ be the set of finite simple reflections and affine simple reflections respectively. 

Let $V = X_* \otimes_{\BZ} \BR$. By definition, alcoves are connected components of $V - \bigcup_{\tilde{\a}}H_{\tilde{\a}}$, where $\tilde{\a}$ runs over the set of affine roots $\Phi_{\aff}$. By convention, the base alcove is defined as 
$$\mathfrak{a} = \{ v\in V\mid 
 0 < \<v,\a \> < 1\text{ for every }\a\in \Phi^+\}.$$

Let $C^+\subseteq V$ be the dominant Weyl chamber and $X_*^+ \subseteq X_*$ be the set of dominant cocharacters. Let $\ell$ be the length function and denote by $\le$ the Bruhat order.



\subsection{Admissible set}\label{sec:1.4}
Let $\mu \in X_*^+$ be a (not necessarily minuscule) dominant cocharacter. Define
$$\Adm(\mu) = \{ w\in \tW \mid w \le t^{ z (\mu)} \text{ for some } z\in W_0\}.$$
Note that $\Adm(\mu)$ has a unique minimal element, which we denote by $\hat{0}$. The maximal elements in $\Adm(\mu)$ are $t^{ \mu'}$, $\mu' \in W_0(\mu)$, the $W_0$-orbit of $\mu$. Then $\Adm(\mu) = \cup_{\mu'\in W_0(\mu)}[\hat{0},t^{\mu'}]$. Note that each $[\hat{0}, t^{\mu'}]$ is a graded poset (see for example \cite[Theorem 2.5.5]{BB}). 

Define $\widehat{\Adm(\mu)}=\Adm(\mu) \sqcup \{\hat 1\}$, where $\hat{1}$ is the added unique maximal element. Then $\widehat{\Adm(\mu)}$ is a graded poset.

The goal of this paper is to prove that $\widehat{\Adm(\mu)}$ is dual EL-shellable (see Theorem \ref{thm:main}). The dual EL-labeling will be defined in \S\ref{sec:2}. And we will prove that it is a dual EL-labeling in \S\ref{sec:4}.

\section{Shellings on the admissible set}\label{sec:2}

\subsection{Reflection order}
Following Dyer \cite[Definition 2.1]{Dyer93}, a total order $\preceq$ on $\Phi^+_{\aff}$ is called a {\it reflection order}, if for any $\tilde\a,\tilde\b \in \Phi^+_{\aff}$ and $a,b\in \BR_{>0}$ such that $\tilde\a \prec \tilde\b$ and $a\tilde\a + b\tilde\b\in\Phi^+_{\aff}$, we have $\tilde\a \prec a\tilde\a + b\tilde\b \prec \tilde\b$. 



For any Bruhat covering $w_2 \gtrdot w_1$ in $\tW$, we associate a label $\eta(w_2 \gtrdot w_1)$, which is defined as the
unique $\tilde\a \in \Phi_{\aff}^{+}$ such that $w_1^{-1} w_2 = w_2^{-1} w_1 = s_{\tilde\a}$.

The following result is established by Dyer, based on the theory of Hecke algebra and $R$-polynomial.
\begin{proposition}[{\cite[Proposition 4.3]{Dyer93}}]\label{prop:Dyer}Let $[w,w']$ be a Bruhat interval in $\tW$. Given any reflection order $\preceq$ on $\Phi_{\aff}^+$, the induced edge labeling $\eta$ on $[w,w']$ is a dual EL-labeling.    
\end{proposition}

Note that Proposition \ref{prop:Dyer} is true for any reflection order $\preceq$ on $\Phi^+_{\aff}$. In order to give a dual EL-labeling on $\widehat{\Adm(\mu)}$, we use a specific reflection order.
\begin{lemma}\label{lem:ref_ord}
There exists a reflection order on $\Phi^+_{\aff}$ such that 
\begin{align*}\tag*{(2.1)}\label{eq:2.1}
(-\a,k) \prec ( \b,k')    
\end{align*}
for any $\a,\b\in\Phi^+$, $k \in \BZ_{\ge1}$ and $k' \in \BZ_{\ge0}$.
\end{lemma}
\begin{proof}The construction is similar to \cite[Proposition 5.2.1]{BB}. We only deal with the case where $\Phi$ is irreducible. The general case is similar.

Choose an ordering $\tilde{\a}_0, \tilde{\a}_1, \ldots, \tilde{\a}_n$ of affine simple roots such that $\tilde{\a}_0 = (-\th, 1)$ where $\th$ is the highest root. Let 
$\CU = \{\sum_{i =0}^n c_i \tilde\a_i \mid 0\le c_i \le 1, \sum_{i =0}^n  c_i = 1\}$. For $ \sum_{i =0}^n  c_i \tilde\a_i \in \CU $ and $  \sum_{i =0}^n  c_i' \tilde\a_i \in \CU $, define $ \sum_{i =0}^n  
 c_i \tilde\a_i <_{\text{lex}} \sum_{i =0}^n   c_i' \tilde\a_i$ if $(c_{0}, \ldots, c_{n}) <_{\text{lex}} (c_{0}', \ldots, c_{n}')$. For any $\tilde\b \in \Phi^+_{\af} $, let $k_{\tilde\b}$ be the unique number in $\BR_{>0}$ such that $k_{\tilde\b} \tilde\b \in  \CU$. For $\tilde\b$, $\tilde\g\in\Phi^+_{\af}$, define 
$$\tilde\b \prec \tilde\g \iff k_{\tilde\b} \tilde\b >_{\text{lex}} k_{\tilde\g} \tilde\g .$$


We first show that $\preceq$ is a reflection order. If $\tilde\a, \tilde\b \in \Phi^+_{\af}$ and $a, b \in \BR_{>0}$ are such that $\tilde\a \prec \tilde\b$ and $a\tilde\a + b\tilde\b \in \Phi_{\aff}^+$, then $k_{a\tilde\a + b\tilde\b}(a\tilde\a + b\tilde\b) = c(k_{\tilde\a} \tilde\a) + (1 - c)(k_{\tilde\b} \tilde\b)$ for some $0 < c < 1$. This, as $k_{\tilde\a} \tilde\a >_{\text{lex}} k_{\tilde\b} \tilde\b$, implies that $k_{\tilde\a} \tilde\a >_{\text{lex}} k_{a\tilde\a + b\tilde\b}(a\tilde\a + b\tilde\b) >_{\text{lex}} k_{\tilde\b} \tilde\b$. This proves that $\prec$ is a reflection order.

We now verify the condition \ref{eq:2.1}. Let $\a,\b\in\Phi^+$, $k\in \BZ_{\ge1}$ and $k' \in \BZ_{\ge0}$. Write $(  -\a,k), ( \b,k')$ and $\th$ as $( -\a ,k) = \sum_{i =0}^n  c_i \tilde{\a}_i$, $( \b,k' ) = \sum_{i =0}^n  c_i' \tilde{\a}_i$ and $\th = \sum_{i=1}^n d_i \a_i$. Then $c_i, c'_i\ge0$ and $d_i>0$ for all $i$. We have $-\a=\sum_{i=1}^n (-c_0 d_i+c_i) \a_i$. Then $c_0 \ge \frac{c_i} {d_i}$ for all $i \ge 1$ and the strict inequality holds for some $i$. Similarly, $c'_0 \le \frac{c'_i} {d_i}$ for all $i \ge 1$ and the strict inequality holds for some $i$. We have $c_0' c_i \le \frac{c_i' c_i }{d_i} \le c_0 c_i'$ for all $i \ge 1$ and the strict inequality $c_0' c_i  < c_0 c_i'$ holds for some $i$. Therefore $c_0 \sum_{i=0}^n c'_i > c'_0 \sum_{i=0}^n c_i$. Hence 
$\frac{c_0}{\sum_{i=0}^n c_i} > \frac{c_0'}{\sum_{i=0}^nc_i'}$. Now by definition, $$\sum_{i=0}^n \frac{c_i}{\sum_{i=0}^n c_i}\tilde{\a}_i >_{\text{lex}} \sum_{i=0}^n \frac{c_i'}{\sum_{i=0}^n c_i'}\tilde{\a}_i$$
and $( -\a, k ) \prec ( \b,k' )$.
\end{proof}

\subsection{Labeling on $\widehat{\Adm(\mu)}$}\label{sec:2.2}

In this subsection, we define an edge labeling $\hat{\eta}$ on $\widehat{\Adm(\mu)}$.

Let $J = I(\mu) = \{ i \in \D_0 \mid \<\mu,\a_i\> = 0 \}$. Let $W^J$ be the set of minimal representatives in $ W_0 / W_J$. Let $\preceq$ be a reflection order on $\Phi^+_{\aff}$ satisfying the condition \ref{eq:2.1}.

We construct a totally ordered set $(\L,\preceq)$ as follows. Associate a symbol $\eta_a$ for each $a \in W^J$ and choose a total order refining the (induced) Bruhat order on $W^J$, that is, a total order $\eta_{a_1} \prec \eta_{a_2} \prec \cdots  \prec \eta_{a_N}$ such that 
\begin{align*}\tag*{(2.2)}\label{eq:2.2}
    a_i < a_j \implies \eta_{a_i} \prec \eta_{a_j}.
\end{align*}
Here, $<$ is the Bruhat order. Let $\L = \Phi^{+}_{\aff} \sqcup \{\eta_a \mid a \in W^J\}$ and impose the relation
\begin{align*}\tag*{(2.3)}\label{eq:2.3}
   (-\a,k') \prec  \eta_a \prec ( \b,k)
\end{align*}
for any $a\in W^J$, $\a,\b\in \Phi^+$, $k'\in \BZ_{\ge1}$ and $k \in \BZ_{\ge0}$. By condition \ref{eq:2.1}, $(\L,\preceq)$ is a well-defined totally ordered set.

Note that the maximal elements in $\Adm(\mu)$ are $t^{a(\mu)}$ for $a \in W^J$. We define an edge labeling $\hat{\eta}: \mathscr{E}(\widehat{\Adm(\mu)}) \to (\L,\preceq)$ as follows. For any edge $w_2 \gtrdot w_1$ in $\Adm(\mu)$, define $\hat{\eta}(w_2 \gtrdot w_1) = \eta(w_2 \gtrdot w_1)$, the unique $\tilde\a \in \Phi^+_{\aff}$ such that $   w_1^{-1} w_2 =w_2^{-1} w_1 = s_{\tilde\a}$. For any edge $\hat{1} \gtrdot t^{a(\mu)}$, define $\eta( \hat{1} \gtrdot t^{a(\mu)}  ) = \eta_a$. 


We think of affine roots $(-\a,k')$ as \textit{negative labels}, affine roots $(\b,k )$ as \textit{positive labels} and $\eta_a$ as \textit{labels near zero}. Then \ref{eq:2.3} can be written as
\begin{align*}
    \text{negative label} \prec \text{label near zero} \prec \text{positive label}.
\end{align*}

The main result of this paper is the following.

\begin{theorem}\label{thm:main}
Let $v \in W_0$. Let 
$$\Adm(\mu)_{\le v}  =  \{w\in \tW\mid w\le t^{v'(\mu)} \text{ for some }v'\le v\}$$
and let $\widehat{\Adm(\mu)_{\le v}} = \{\hat{1}\}\sqcup \Adm(\mu)_{\le v}$. Then the restriction $\hat{\eta}_{\le v}$ of $\hat{\eta}$ on $\widehat{\Adm(\mu)_{\le v}}$ is a dual EL-labeling. 
\end{theorem}

By Theorem \ref{thm:BW} and Lemma \ref{lem:coatom}, we have
\begin{corollary}\label{cor:CM}
Let $v \in W_0$. Then $\Adm(\mu)_{\le v}$ is $\<\mu,2\rho\>$-Cohen-Macaulay.
\end{corollary}

Let $v$ be the maximal element in $W_0$. We deduce the following.

\begin{corollary}\label{cor:main}
The edge labeling $\hat{\eta}$ is a dual EL-labeling on $\widehat{\Adm(\mu)}$ and $\Adm(\mu)$ is $\<\mu,2\rho\>$-Cohen-Macaulay.
\end{corollary}

\begin{remark}
The dual EL-labeling $\hat{\eta}$ induces a recursive coatom ordering $t^{a_1(\mu)}$, $t^{a_2(\mu)}$, $\ldots$, $t^{a_{\ell}(\mu)}$ of $\widehat{\Adm(\mu)}$. This ordering is nothing but the ordering we choose in \ref{eq:2.2}, which can be any ordering on $W^J$ refining the Bruhat order. Indeed, this follows from the proof of Theorem \ref{thm:BW} (see \cite[Proof of Theorem 3.2]{BW83}).
\end{remark}

The proof of Theorem \ref{thm:main} will be given in \S\ref{sec:4}. In \S\ref{sec:3}, we give some technical preparation for the proof.

\section{Quantum Bruhat graph and Bruhat order on $\tW$}\label{sec:3}

In \cite{Sch24}, Schremmer obtained a nice characterization of Bruhat covering in $\tW$ using the quantum Bruhat graph and acute presentation. We use this characterization to give an explicit description of the top two layers of $\Adm(\mu)$ (see Proposition \ref{prop:top-two}).

\subsection{Quantum Bruhat graphs}\label{sec:3.1} 
We first recall the quantum Bruhat graph introduced by Fomin, Gelfand and Postnikov in \cite{FGP97}. 

Let $\CR$ be a based root datum and $\QBG(\CR)$ be the associated quantum Bruhat graph. The set of vertices of $\QBG(\CR)$ is the set $W_0$. The edges of $\QBG(\CR)$ are of the form $w\rightarrow ws_\a$ for $w\in W_0$ and $\a\in \Phi^+$ whenever one of the following conditions is satisfied:
\begin{itemize}
\item $ w s_\a \gtrdot w $ or
\item $\ell(ws_\a) = \ell(w)+1-\langle \a^\vee,2\rho\rangle$.
\end{itemize}
The edges satisfying the first condition are called \emph{Bruhat} edges and the edges satisfying the second condition are called \emph{quantum} edges. The weight of a Bruhat edge is defined to be zero. The weight of a quantum edge $x\rightarrow xs_\a$ is defined as the coroot $\a^{\vee}$. The weight of a path is defined as the sum of the weights of all edges of the path. It is easy to see that $\QBG(\CR)$ is path-connected. 

We shall use $\rightharpoonup$ and $\rightharpoondown$ to denote Bruhat edges and quantum edges, respectively, emphasizing that the edges are going up or down. We also use $\to$ for both Bruhat and quantum edges. 

The following is the quantum Bruhat graph of the type $A_2$ Weyl group:

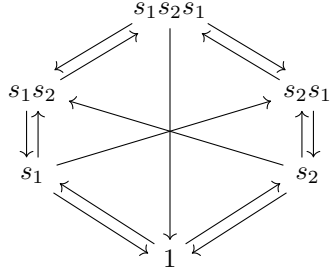
\def\qecolor{dashed}
\def\seshift{0.5ex}
\def\quantumEdge{dashed}
\def\shortEdgeShiftRight{0.5ex}
\begin{align*}
\begin{tikzcd}[ampersand replacement=\&,column sep=2em]
\&s_1 s_2 s_1\ar[ddd,]
\ar[dl,shift right=\seshift]\ar[dr,shift left=\seshift]\\
s_1 s_2\ar[ur,shift right=\seshift]\ar[d,shift right=\seshift]\&\&
s_2 s_1\ar[ul,shift left=\seshift]\ar[d,shift left=\seshift]\\
s_1\ar[u,shift right=\seshift]\ar[urr]\ar[dr,shift right=\seshift]\&\&s_2\ar[u,shift left=\seshift]\ar[ull]\ar[dl,shift left=\seshift]\\
\&1\ar[ru,shift left=\seshift]\ar[lu,shift right=\seshift]
\end{tikzcd}
\end{align*}

We recall some basic properties of the quantum Bruhat graph.
\begin{lemma}[{\cite[Lemma 1]{Pos05}}]\label{lem:Pos}
Let $u,v\in W_0$.
\begin{enumerate}
    \item All shortest paths from $u$ to $v$ in $\mathrm{QBG}(\CR)$ have the same weight. We denote this weight by $\wt(u, v)$.
    \item Any path from $u$ to $v$ has weight $\ge$ $\wt(u,v)$, the equality holds if and only if the path is shortest.
    \item $\wt(u ,v) = 0$ if and only if $u \le v$.
\end{enumerate}
\end{lemma}
By definition of the weight function, we have 
\begin{align*}\tag*{(3.1)}\label{eq:3.1}
\wt(u,u'') \le \wt(u,u')+ \wt(u',u'') \text{ for any } u,u',u''\in W_0.
\end{align*}

There is a duality of $\QBG(\CR)$ given by 
\begin{align*}\tag*{(3.2)}\label{eq:3.2}
 u \mapsto w_0 u.   
\end{align*}
Here, $w_0$ is the longest element in $W_0$. Indeed, it is easy to see that the map sends a Bruhat edge to a Bruhat edge and a quantum edge to a quantum edge.

We end this subsection with the following property, which follows from \cite[Corollary 3.4]{HY24} and the duality \ref{eq:3.2}. It is an essential ingredient in the proof of Theorem \ref{thm:main}. 

\begin{proposition}\label{prop:min}
Let $\g\in \sum_{i\in \BS_0} \BZ_{\ge0} \a_i^{\vee}$ and $v\in W_0$. Then 
\begin{enumerate}

\item the set $\{ u \in W_0 \mid \wt(u ,v) \le \g\}$ contains a unique maximal element;

\item the set $\{ u \in W_0 \mid \wt(v,u) \le \g\}$ contains a unique minimal element. 
\end{enumerate}
\end{proposition}

\subsection{Acute presentation}\label{sec:3.2}

By abuse of notation, for $\a\in\Phi$, define
\begin{align*}
    \Phi^-(\a) = \begin{cases}
        1, & \text{ if }\a\in\Phi^-;\\
        0, & \text{ if }\a\in\Phi^+.
    \end{cases}
\end{align*}
For $\a\in\Phi$ and $k \in \BZ$, we have $(\a,k) \in \Phi^+_{\af}$ if and only if $k \ge  \Phi^-(\a) $. For $z \in W_0$, we have $\ell(z) = \sum_{\a\in\Phi^+} \Phi^-(z\a)$.

\begin{definition}
Let $w =xt^{\l}y \in\tW$ with $x,y\in W_0$ and $\l\in X_*$. We say $ xt^{\l}y$ is an \emph{acute presentation} of $w$ if 
$$ \Phi^-(x(\a)) + \<\l,\a \> - \Phi^-(y^{-1}( \a)) \ge0 $$
for all $\a\in\Phi^+$.
\end{definition}
The notion of acute presentation is a reformulation of the notion of \emph{length-positive element} in \cite[Definition 2.2]{Sch24}. In the rest of this subsection, we give a geometric explanation of the notion of acute presentation via acute cone introduced by Haines and Ngo in \cite[\S5]{HN02}.

Let $H = H_{\tilde\a}$ be the root hyperplane corresponding to $\tilde\a = (\a,k)$ (see \S\ref{sec:1.3} for the notations). Let $z\in W_0$. Assume that $\a\in z(\Phi^+)$ (otherwise replace $\tilde\a$ with $-\tilde\a$). Define $H^{z+} = \{v\in V \mid \<v,\a\> > -k\}$, that is, the connected component of $V - H$ that contains any sufficiently deep alcoves in the Weyl chamber $z(C^+)$ (recall that $C^+$ is the dominant chamber). The \emph{acute cone in the $z$-direction} is defined to be $$\CC(\mathfrak{a},z)=\{w \in  \tW \mid w(\mathfrak{a}) \subseteq H^{z+} \text{ for all root hyperplanes }H \text{ with } \mathfrak{a} \subset H^{z+}\}.$$

We have the following natural bijection between the acute presentations of $w$ and the acute cones containing $w$. 

\begin{proposition}\label{prop:acute}
Let $w =xt^{\l}y \in\tW$ with $x,y\in W_0$ and $\l\in X_*$. Then $ xt^{\l}y$ is an acute presentation of $w$ if and only if $w \in \CC(\mathfrak{a},x)$.
\end{proposition}
\begin{proof}
By definition, we have 
$$\mathfrak{a} = \{ v\in V \mid -\Phi^-(\b) < \<v,\b\> <  1-\Phi^-(\b) \text{ for every }\b\in\Phi^+   \}.$$
Then by definition, $w\in\CC(\mathfrak{a},x)$ if and only if
$$w(\mathfrak{a}) \subseteq \bigcap_{  \a\in\Phi^+  }   \{v\in V\mid \<v,x(\a)\> \ge -\Phi^-(x\a)   \}  .$$
For any $v \in V$ and $\a \in \Phi^+$, we have $\<w(v) , x (\a)\> =\< x(\l) + xy( v)  ,x(\a)\> = \<\l,\a\>+ \<v,y^{-1}(\a)\>  $. Hence $w\in\CC(\mathfrak{a},x)$ if and only if 
$$\<\l,\a\>  - \Phi^-(y^{-1}(\a))   \ge -\Phi^-(x(\a)) $$
for any $\a\in \Phi^+$.
\end{proof}


Let $w = xt^{\l}y$ with $x,y\in W_0$ and $\l\in X_*$. From the proof of Proposition \ref{prop:acute}, we see that the absolute value of $\Phi^-(x(\a)) + \<\l,\a \> - \Phi^-(y^{-1}( \a))$ is exactly the number of root hyperplanes $H_{(x(\a),k)}$ ($k\in\BZ$) separating $w(\mathfrak{a})$ and $\mathfrak{a}$. Hence, \begin{align*}
\ell(w) &= \sum_{\a\in\Phi^+}  \left\vert \Phi^-(x(\a)) + \<\l,\a \> - \Phi^-(y^{-1}( \a))  \right \vert  \\
&\ge \sum_{\a\in\Phi^+}   \Phi^-(x(\a)) + \<\l,\a \> - \Phi^-(y^{-1}( \a)) \\
&= \ell(x ) + \<\l,2\rho\> - \ell(y),
\end{align*}
the equality holds if and only if $ xt^{\l}y$ is an acute presentation of $w$ (cf. \cite[Corollary 2.11]{Sch23}).

Note that any element $w\in\tW$ can be written uniquely as $w = x_0t^{\l_0}y_0 $ with $x_0, y_0 \in W_0$, $\l_0 \in X_*^+$ such that $t^{\l_0} y_0$ is of minimal length in its coset $W_0 \backslash \tW$ (see for example \cite[\S9.1]{He14}). We call it the \textit{standard presentation of $w$}. In this case, it is easy to see that $w(\mathfrak{a})$ lies in the Weyl chamber $x_0(C^+)$. Then $w\in \CC(\mathfrak{a},x_0)$ and $x_0 t^{\l_0} y_0$ is an acute presentation of $w$. In general, $w$ may lie in more than one acute cone. Note that $w$ lies in only one acute cone if and only if $w(\mathfrak{a})$ lies in the shrunken Weyl chamber, or equivalently, the lowest two-sided Kazhdan-Lusztig cell of $\tW$ (see \cite[Proposition 2.15]{Sch23}). We don't need this result in this paper.

We end this subsection with a full description of acute presentations of translation elements.
\begin{lemma}\label{lem:trans}
Let $z \in W_0$ and $\l_0\in X_*^+$. Then the acute presentations of $t^{z(\l_0)}$ are $z u t^{\l_0} u^{-1} z$, where $u$ runs over $W_J$.
\end{lemma}
\begin{proof}
Let $u \in W_0$. We have $t^{z(\l_0)} = z u t^{u^{-1}(\l_0)} u^{-1} z^{-1}$. Then
\begin{align*}
    &z u t^{u^{-1}(\l_0)} u^{-1} z^{-1} \text{ is an acute presentation of }t^{z(\l_0)}\\
    \iff&  \Phi^-(zu(\a)) + \<u^{-1}(\l_0), \a\> - \Phi^-( zu(\a) ) \ge 0 \text{ for any }\a\in \Phi^+\\
    \iff&u^{-1}(\l_0)\text{ is dominant}\\
    \iff & u\in W_J. \qedhere
\end{align*}
\end{proof}

\subsection{A characterization of Bruhat covering}\label{sec:3.3}

The following important result is due to Schremmer.
\begin{theorem}[{\cite[Theorem 4.2]{Sch24}}]\label{thm:Felix}
Let $w\in\tW$ and let $  x t^{\l}y$ be an acute presentation of $w$. Let $w' \in \tW$. Then $w\le w'$ if and only if there exists $x',y' \in W_0$ and $\l'\in X_*$ with $w' = x' t^{\l'} y'$ such that
$$\wt(x,x') + \wt(y'^{-1},y^{-1}) \le \l' - \l.$$
\end{theorem}

As an application of Theorem \ref{thm:Felix}, Schremmer obtained the following explicit description of Bruhat covering, generalizing a result of Lam and Shimozono \cite[Proposition 4.4]{LS10}, in which $\l$ is assumed to be very regular.

\begin{proposition}[{\cite[Proposition 4.5]{Sch24}}]\label{prop:Felix}
Let $w\in\tW$ and let $ x t^{\l}y$ be an acute presentation of $w$. Let $\tilde{\a} \in \Phi^+_{\aff}$ and $w' = w s_{\tilde{\a}}$. Then $ w' \gtrdot w$ if and only if there is a root $\a\in\Phi^+$ satisfying one of the following conditions.
\begin{enumerate}[(i)]

\item There exists a Bruhat edge $y^{-1}s_\a \rightharpoonup y^{-1}$, $\tilde{\a} = (-y^{-1}\a,0)$ and $ x t^{\l } s_{\a}y$ is an acute presentation of $w'$.

\item There exists a quantum edge $y^{-1}s_\a \rightharpoondown y^{-1}$, $\tilde{\a} = (-y^{-1}\a, 1)$ and $ x t^{\l + \a^{\vee} } s_{\a}y$ is an acute presentation of $w'$.

\item There exists a Bruhat edge $x \rightharpoonup x s_\a$, $\tilde{\a} = ( y^{-1}\a, \<\l,\a\> )$ and $ x s_\a t^{\l }  y$ is an acute presentation of $w'$.

\item There exists a quantum edge $x \rightharpoondown x s_\a $, $\tilde{\a} = ( y^{-1}\a, \<\l,\a\> + 1)$ and $  x s_\a t^{\l + \a^{\vee} }  y$ is an acute presentation of $w'$.

\end{enumerate}
\end{proposition}
The root $\a \in\Phi^+$ in Proposition \ref{prop:Felix} depends on the choice of the acute presentation $  xt^{\l}y$ of $w$. It is not hard to see that the four cases in Proposition \ref{prop:Felix} are exclusive. 

In case (i) and (ii) of Proposition \ref{prop:Felix}, we say the Bruhat covering $w'\gtrdot w$ is a ``$y$-move", as the acute presentation of $w'$ is obtained from $xt^{\l}y$ by changing $y$. In case (iii) and (iv) of Proposition \ref{prop:Felix}, we say the Bruhat covering $w'\gtrdot w$ is a ``$x$-move". We point out that the ``move type" also depends on the choice of the acute presentation $x t^{\l} y$ of $w$.

\begin{remark}\label{rmk:pos-neg}
Let $\hat{\eta} : \mathscr{E}(\widehat{\Adm(\mu)}) \to \L$ be the edge labeling constructed in \S\ref{sec:2.2}. Then in case (i) of Proposition \ref{prop:Felix}, $\tilde{\a}$ is a positive label since $y^{-1} \a  \in\Phi^+$, while in case (ii), $\tilde{\a}$ is a negative label since $y^{-1}\a \in \Phi^-$. In case (iii) or (iv), it is not clear whether $\tilde{\a}$ is a positive or a negative label.
\end{remark}

\subsection{From chains in $\Adm(\mu)$ to paths in $\text{QBG}(\CR)$}\label{sec:3.4}   

Let $w\in\Adm(\mu)$ and $a\in W^J$ with $w < t^{a(\mu)}$. Let $  x t^{\l}y$ be an acute presentation of $w$. Let 
\begin{align*}
 t^{a(\mu)} \gtrdot   w_r \gtrdot w_{r-1} \gtrdot \cdots \gtrdot w_1\gtrdot w 
\end{align*}
be a maximal chain in the Bruhat interval $[w,t^{a(\mu)} ]$. 

Apply Proposition \ref{prop:Felix} to the Bruhat covering $w_1 \gtrdot w$. We obtain an acute presentation $ x_1 t^{\l_1} y_1$ of $w_1$, where, in case (i) or (ii) of Proposition \ref{prop:Felix}, we have $x = x_1$ and there is an edge $y^{-1} \to y_1^{-1}$ in $\QBG(\CR)$ while in case (iii) or (iv) of Proposition \ref{prop:Felix}, we have $y = y_1$ and there is an edge $x \to x_1 $ in $\QBG(\CR)$. Apply Proposition \ref{prop:Felix} repeatedly to the above chain and use Lemma \ref{lem:trans}, we obtain an acute presentation $ x_i t^{\l_i} y_i $ of $w_i$ for each $i$ and an acute presentation $ au t^{\mu} u^{-1}a^{-1}$ of $t^{a(\mu)}  $ ($u\in W_J$). Then we obtain a path
\begin{align*}\tag*{(3.3)}\label{eq:3.3}
x \to x' \to \cdots \to x'' \to au \to y''^{-1} \to \cdots  \to y'^{-1} \to y^{-1}
\end{align*}
in $\QBG(\CR)$. The weight of the path contributes to the change of the translation parts (case (ii), (iv) of Proposition \ref{prop:Felix}), which is equal to $\mu-\l$. In particular, we have
\begin{align*}\tag*{(3.4)}\label{eq:3.4}
\wt(x,y^{-1})\le\mu-\l.    
\end{align*}
Note that the path \ref{eq:3.3} may not be a shortest path from $x$ to $y^{-1}$ and we may have $\wt(x,y^{-1}) < \mu - \l$ in general.

    


\subsection{Top two layers of $\Adm(\mu)$}

Let $\hat{\eta}:\mathscr{E}(\widehat{\Adm(\mu)}) \to (\L,\preceq)$ be the edge labeling constructed in \S\ref{sec:2.2}. We now give an explicit description of the top two layers of $\Adm(\mu)$.

\begin{proposition}\label{prop:top-two}
Let $w \in \Adm(\mu)$ with $ \ell(w) = \<\mu,2\rho\> -1$. Let $z_1t^{\l}z_2^{-1}$ be an acute presentation of $w$. Then,
\begin{enumerate}
\item there is an edge $z_1\to z_2 $ in $\text{QBG}(\CR)$, $\l = \mu - \wt(z_1,z_2)$ and $\wt(z_1,z_2)\notin \Phi_J$;

\item $w$ is covered by exactly two elements $t^{z_1(\mu)}$ and $t^{z_2(\mu)}$;

\item let $z = \min(z_1,z_2)$ and $z' = \max(z_1,z_2)$, then the chain $ \hat{1}\gtrdot t^{z(\mu)}\gtrdot w $ is label-increasing and the chain $ \hat{1}\gtrdot t^{z'(\mu)}\gtrdot w $ is label-decreasing.
\end{enumerate}

\end{proposition}

\begin{proof}Let $t^{\mu'}\gtrdot w$ be a Bruhat covering, where $\mu' \in W_0(\mu)$, the $W_0$-orbit of $\mu$.

Apply Proposition \ref{prop:Felix} and Lemma \ref{lem:trans}. There exists $\a\in\Phi^+$ such that one of the following conditions happens.

\begin{enumerate}[(i)]
\item $z_2 = z_1 s_\a$, $\l=\mu$, $\mu' = z_1(\mu)$ and there is a Bruhat edge $ z_1 \rightharpoonup z_2 $. In this case, $\hat{\eta}(t^{\mu'} \gtrdot w) = (z_1 (\a) , 0)$ is a positive label.

\item $z_2 = z_1 s_\a$, $\l=\mu - \a^{\vee}$, $\mu' = z_1(\mu)$, there is a quantum edge $ z_1 \rightharpoondown z_2 $. In this case, $\hat{\eta}(t^{\mu'} \gtrdot w) = ( z_1(\a) , 1)$ is a negative label.

\item $z_2 = z_1 s_\a$, $\l=\mu$, $\mu' = z_2 (\mu)$, there is a Bruhat edge $ z_1 \rightharpoonup z_2$. In this case, $\hat{\eta}(t^{\mu'} \gtrdot w) = (-z_1(\a), \< \mu,\a\>)$ is a negative label.
    
\item $z_2 = z_1 s_\a$, $\l=\mu - \a^{\vee}$, $\mu' = z_2 (\mu)$, there is a quantum edge $ z_1 \rightharpoondown z_2$. In this case, $\hat{\eta}(t^{\mu'} \gtrdot w) = (-z_1(\a), \< \mu,\a\>-1)$ is a positive label.
    
\end{enumerate}

Note that all statements of Proposition \ref{prop:top-two} except that $\a\notin\Phi_J$ follow directly from the above explicit description and the definition of $\hat{\eta}$.

In case (i) and (iii), as $z_1 t^{\mu} s_\a z_1^{-1} $ is an acute presentation of $w$, we get
$$  \Phi^- (z_1 (\a)) + \< \mu , \a\> - \Phi^-( z_1 s_\a (\a)) \ge 0.$$
Since $z_1(\a) \in \Phi^+$, we have $0 + \<\mu,\a\> -1  \ge 0$ and hence $\a\notin\Phi_J$. 

In case (ii) and (iv), as $z_1 t^{\mu-\a^{\vee}} s_\a z_1^{-1} $ is an acute presentation of $w$, we get
$$   \Phi^- (z_1(\a)) + \< \mu -\a^{\vee}, \a\> - \Phi^-( z_1 s_\a (\a)) \ge 0.$$
Since $z_1(\a) \in \Phi^-$, we have $1 + \<\mu,\a\> - 2 + 0 \ge 0$ and hence $\a\notin\Phi_J$. 
\end{proof}

Proposition \ref{prop:top-two} (3) is essential for the proof of Theorem \ref{thm:main}. And this is the reason we construct the edge labeling $\hat{\eta}$ in this way (see \ref{eq:2.2} and \ref{eq:2.3}).

\begin{remark}
Proposition \ref{prop:top-two} (2) was first established by Haines in \cite[Proposition 8.7]{Ha05} (see also \cite[Remark 3.2]{GL24}). We give a different proof here. We also thank Haines for pointint out loc.cit. to us. 
\end{remark}
\subsection{A by-product on quantum Bruhat graph}


As a by-product, we obtain the following property of the quantum Bruhat graph, which is of independent interest. This result will not be used in this paper. The reader can skip this subsection.
\begin{proposition}\label{prop:downup}
Let $u,v \in W_0$ with $u\ne v$. Then there exists a shortest path from $u$ to $v$ of ``down-up type", that is, quantum edges come first and Bruhat edges come later. Similarly, there exists a shortest path from $u$ to $v$ of ``up-down type". 
\end{proposition}

\begin{proof}
Let $\l$ be a very dominant regular cocharacter, that is, $\<\l,\a_i\>$ is sufficiently large for all $i\in \D_0$. Let $w = u t^{\l} v^{-1}$ and $w' = u t^{\l+\wt(u,v)} u^{-1}$.

Note that the edge labeling $\eta$ is a dual EL-labeling on $[w,w']$ by Proposition \ref{prop:Dyer}. Let 
\begin{align*}\tag*{(3.5)}\label{eq:3.5}
    w'\gtrdot w_r\gtrdot \cdots\gtrdot w_1=w
\end{align*}
be the unique label-increasing maximal chain from $w'$ to $w$. As in \ref{eq:3.3}, the chain \ref{eq:3.5} gives rise to a path
$$p: u \to u'\to \cdots\to u'' \to u \to v'' \to \cdots \to v' \to v$$
in $\text{QBG}(\CR)$ whose weight equals $\wt(u,v)$. By Lemma \ref{lem:Pos} (2), the path is shortest. Hence, the subpath $u \to u'\to \cdots\to u''$ is empty. Then each edge in \ref{eq:3.5} is a ``y-move", and each positive label edge corresponds to a Bruhat edge of $p$ while each negative label edge corresponds to a quantum edge of $p$.\footnote{these are cases (i) and (ii) in Proposition \ref{prop:top-two} respectively.}

Since \ref{eq:3.5} is label-increasing, negative labels come first and positive labels come later. Hence $p$ is of ``down-up type". Reversing the reflection ordering $\preceq$, we can prove that there exists a shortest path from $u$ to $v$ of ``up-down type".
\end{proof}

In the case where $v = 1$, Proposition \ref{prop:downup} implies that there exists a shortest path from $u$ to $1$ of ``down type", that is, consisting only of quantum paths. This recovers \cite[Proposition 4.11]{MV20}. More generally, let $v \in W_0$ and $\g\in\sum_{i\in\BS_0}\BZ_{\ge0}\a_i^{\vee}$. Let $u$ be the unique maximal element with $\wt(u,v)\le\g$ as in Proposition \ref{prop:min}. Using Proposition \ref{prop:downup}, we conclude that there exists a shortest path from $u$ to $v$ of ``down type".

\section{Proof of Theorem \ref{thm:main}}\label{sec:4}


Let $\hat{\eta}:\mathscr{E}(\widehat{\Adm(\mu)}) \to (\L,\preceq)$ be the edge labeling on $\widehat{\Adm(\mu)}$ constructed in \S\ref{sec:2.2}. We first prove Corollary \ref{cor:main}, that is, $\hat{\eta}$ is a dual EL-labeling on $\widehat{\Adm(\mu)}$. 

We need to prove that for any interval $[w,w']$ of $\widehat{\Adm(\mu)}$, there is a unique maximal chain in $[w,w']$ which is label-increasing, and this chain is lexicographically minimal among all maximal chains in $[w,w']$. By Proposition \ref{prop:Dyer}, it suffices to prove it for the interval $[w,\hat{1}]$ for any $w \in \Adm(\mu) $ with $\ell(w) < \<\mu,2\rho\>$.

\subsection{The set $\Sigma_w^J$}\label{sec:4.1}
Let $w\in\Adm(\mu)$ with $\ell(w) < \<\mu,2\rho\>$ and let $ x t^{\l}y$ be an acute presentation of $w$. Let $J  = \{i\in\D_0 \mid \<\mu,\a_i\>=0\}$. Define 
\begin{align*}
\Sigma_w &= \{ z \in W_0 \mid \wt(x, z ) + \wt( z,y^{-1}) \le \mu - \l\} ;\\ 
\Sigma_w^J &= \{ a \in W^J \mid au \in \Sigma_w \text{ for some } u\in W_J\}.
\end{align*}
We claim that
\begin{align*}
    \Sigma_w^J = \{a \in W^J \mid w \le  t^{a(\mu)}\}.
\end{align*}
Indeed, the $``\subseteq"$ direction follows from Theorem \ref{thm:Felix} and the $``\supseteq"$ direction follows from the construction of the path \ref{eq:3.3}. 

\begin{lemma}\label{lem:min}
The set $\Sigma_w$ has a unique minimal element $z_{\min}$. The set $\Sigma_w^J$ has a unique minimal element $a_{\min}$. Moreover, we have $z_{\min} \in a_{\min}W_J$.
\end{lemma}

\begin{remark}
We can prove similarly that $\Sigma_w$ has a unique maximal element $z_{\max}$. Note also that the set $\Sigma_w$ may not be an interval.
\end{remark}

\begin{proof}
By Proposition \ref{prop:min} (2), the set $\{ z \in W_0 \mid  \wt(x,z) \le \mu - \l \}$ has a unique minimal element $z_{\min}$. We claim that $z_{\min}$ is the unique minimal element in $\Sigma_w$.

It is clear that $\Sigma_w \subseteq \{z \in W_0 \mid \wt(x,z) \le \mu - \l\}$. To prove the claim, it suffices to prove that $z_{\min} \in \Sigma_w$. By \ref{eq:3.4}, we have $\wt(x , y^{-1}) \le \mu -\l$. Then $y^{-1} \ge z_{\min}$ by definition of $z_{\min}$. Hence $\wt(z_{\min}, y^{-1}) = 0 $ by Lemma \ref{lem:Pos} (3). It follows that $z_{\min} \in \Sigma_w$. This proves the claim.

Let $a_{\min}\in W^J$ be such that $z_{\min } \in a_{\min} W_J $. By definition of $\Sigma_w^J$, we see that $a_{\min}$ is the unique minimal element in $\Sigma_w^J$.
\end{proof}

We point out that the set $\Sigma_w$ and the element $z_{\min}$ depend on the choice of the acute presentation $xt^{\l}y$ of $w$. However, since $\Sigma_w^J= \{ a \in W^J \mid w \le t^{a(\mu)}\}$, both the set $\Sigma_w^J$ and the element $a_{\min}$ are independent of the choice of acute presentations. 

\subsection{Lexicographically minimal chain}
For each $a \in \Sigma_w^J$, consider the Bruhat interval $[w , t^{a(\mu)}]$. By Proposition \ref{prop:Dyer}, there is a unique label-increasing maximal chain 
$$p_{a,w}: t^{a(\mu)} \gtrdot w_{a,r} \gtrdot w_{a,r-1} \gtrdot \cdots \gtrdot w_{a,1} = w$$
in $[w , t^{a(\mu)}]$ ($r = \<\mu,2\rho\> - \ell(w)  $), and $p_{a,w}$ is lexicographically minimal among all maximal chains in $[w , t^{a(\mu)}]$. Consider the concatenation 
\begin{align*}\tag*{(4.1)}\label{eq:4.1}
    \hat{1} \gtrdot t^{a_{\min}(\mu)} \gtrdot w_{a_{\min},r} \gtrdot w_{a_{\min},r-1} \gtrdot \cdots \gtrdot w_{a_{\min},1}=   w.
\end{align*}

Since $a_{\min}$ is the minimal element in $W^J$ with $w \le t^{a_{\min}(\mu)}$, by Proposition \ref{prop:top-two} (3), the subchain $\hat{1} \gtrdot t^{a_{\min}(\mu)} \gtrdot w_{a_{\min},r}$ is label-increasing. Hence, the concatenation \ref{eq:4.1} is label-increasing. By condition \ref{eq:2.2} of $\hat{\eta}$, the chain \ref{eq:4.1} is lexicographically minimal among all maximal chains in $[w, \hat{1}]$.

\subsection{Uniqueness of the label increasing chain} To prove Corollary \ref{cor:main}, it remains to prove that \ref{eq:4.1} is the unique label-increasing maximal chain in $[w,\hat{1}]$. Let $a \in \Sigma_w^J$ with $a\ne a_{\min}$. Since $p_{a,w}$ is the unique label-increasing maximal chain in $[w,t^{a(\mu)}]$, it suffices to prove that $\hat{1}\gtrdot t^{a(\mu)} \gtrdot w_{a,r}$ is not label-increasing, or equivalently, $ t^{a(\mu)} \gtrdot w_{a,r}$ has a negative label (recall that $\hat{\eta}(\hat{1} \gtrdot t^{a(\mu)})$ is a label near zero, see \ref{eq:2.3}).

Assume otherwise $ t^{a(\mu)} \gtrdot w_{a,r}$ has a positive label. We shall construct an element $w'$ with 
\begin{align*}\tag*{(*)}\label{*}
t^{a(\mu)} \gtrdot w' \ge w \text{ and } t^{a(\mu)}  \gtrdot w' \text{ has a negative label.}
\end{align*}
Then it contradicts the fact that $p_{a,w}$ is lexicographically minimal among all maximal chains in $[w,t^{a(\mu)]}]$, and this completes the proof of Corollary \ref{cor:main}.

See Figure 1 for the relationship between the elements in $\widehat{\Adm(\mu)}$ being considered.

\begin{figure}[h]
\centering
\begin{tikzpicture}
    \node (w) at (0,0) {$w = xt^{\l}y$};
    \node (wamin1) at (-2,1) {$w_{a_{\min},1}$};
    \node (wa1) at (0,1) {$w_{a,1}$};
    \node (waminr) at (-3,3) {$w_{a_{\min},r}$};
    \node (w') at (3,3) {$w' = z'u^{-1} t^{\mu} a^{-1}$};
    \node (war) at (0,3) {$w_{a,r}$};
    \node (tamin) at (-3,4) {$t^{a_{\min}(\mu)}$};
    \node (tz') at (0,4) {$t^{z'(\mu)}$};
    \node (ta) at (3,4) {$t^{a(\mu)}$};
    \node (tdot) at (-2,4) {$\cdots$};
    \node (tdot') at (-1,4){$\cdots$};
    \node (hat1) at (0,5) {$\hat{1}$};
    \draw[-]  (w) -- (wamin1);
    \draw[-]  (w) -- (wa1);
    \draw[-]  (w) -- (w');
    \draw[-]  (wamin1) -- (waminr);
    \draw[-]  (wa1) -- (war);
    \draw[-]  (war) -- (ta);
    \draw[-]  (waminr) --  (tamin);
    \draw[-]  (waminr) --  (tdot);
    \draw[-]   (war) -- (tdot');
    \draw[-]  (w') --(tz');
    \draw[-]  (w') --(ta);
    \draw[-]  (tamin) -- (hat1);
    \draw[-]  (tz') -- (hat1);
    \draw[-]  (tdot) -- (hat1);
    \draw[-]  (tdot') -- (hat1);
    \draw[-]  (ta) -- (hat1);
\end{tikzpicture}
\caption{}
\end{figure}

The construction of $w'$ is the most technical part of this paper. By Proposition \ref{prop:top-two}, to construct $w'$, we need to find an edge $z_1\to z_2$ in $\text{QBG}(\CR)$ with $z_1  W_J \ne z_2 W_J$ and $\max(z_1 W_J, z_2 W_J) = a W_J$ such that $ z_1 t^{\mu - \wt(z_1,z_2)}z_2^{-1}$ is an acute presentation and $z_1 t^{\mu - \wt(z_1,z_2)}z_2^{-1}\ge w$.

As in \S\ref{sec:3.4}, the chain $p_{a,w}$ gives rise to a path
\begin{align*}
x \to x'  \to \cdots \to x'' \to au \to y''^{-1} \to \cdots  \to y'^{-1} \to y^{-1}
\end{align*}
in $\QBG(\CR)$ whose weight is equal to $\mu - \l$ with $u\in W_J$. Since $p_{a,w}$ is label-increasing and $ t^{a(\mu)} \gtrdot w_{a,r}$ has a positive label by assumption, all edges in $p_{a,w}$ have positive labels. By Remark \ref{rmk:pos-neg}, case (ii) of Proposition \ref{prop:Felix} does not occur in $p_{a,w}$. Hence, each edge of the subpath $au \to y''^{-1} \to \cdots  \to y'^{-1} \to y^{-1}$ comes from case (i) of Proposition \ref{prop:Felix}. In particular, they are all Bruhat edges and hence $au \le  y^{-1}$. Therefore, $\wt(au,y^{-1})=0$ by Lemma \ref{lem:Pos} (3).

Note that $au \in \Sigma_w$ by definition. Recall that $z_{\min} = \min (\Sigma_w)$. Then $au \ge z_{\min}$. Since $z_{\min} \in a_{\min} W_J$ and $a>a_{\min}$ by assumption, $au > z_{\min}$. By \cite[\S7]{LNSSS}, there exists a chain
\begin{align*}
    au \gtrdot z' \gtrdot \cdots \gtrdot z'' \gtrdot z_{\min}
\end{align*}
in $W_0$ such that $z' \notin aW_J$. By \ref{eq:3.1} and the definition of $z_{\min}$, we have
\begin{align*}\tag*{(4.2)}\label{eq:4.2}
    \wt(x, z') + \wt( au ,y^{-1}) \le \wt(x,z_{\min}) + 0 \le \mu -\l.
\end{align*}

Set $w' = z' t^{\mu} u^{-1}a^{-1}  =  z' u^{-1} t^{\mu}  a^{-1}$. 
By Theorem \ref{thm:Felix}, \ref{eq:4.2} implies that $w \le w'$. On the other hand, since $z'\lessdot au$, $u\in W_J$ and $a\in W^J$, we must have $z'  u^{-1} \lessdot a$. Note that $ z' u^{-1} t^{\mu} a^{-1}$ is the standard presentation of $w'$.\footnote{Note that $z'  t^{\mu} u^{-1} a^{-1}$ is also an acute presentation of $w'$. To prove this, it suffices to exclude the case where $\b\in \Phi_J^+$, $z'(\b)\in\Phi^+$ and $au(\b) \in\Phi^-$. However, using the standard presentation $z' u^{-1} t^{\mu} a^{-1}$ is more straightforward.} By Proposition \ref{prop:Felix} (3), we have a Bruhat covering $t^{a(\mu)}\gtrdot w'$ whose label is negative. Hence, $w'$ satisfies the desired property \ref{*}. This completes the proof of Corollary \ref{cor:main}.

See Figure 2 for the elements and paths in $\text{QBG}(\CR)$ being considered in the proof.

\begin{figure}[h]
\centering

\centering
\begin{tikzpicture}
    \node (y') at (0,2.5) {$y^{-1}$};
    \node (x) at (2,2) {$x$};
    \node (au) at (0,1) {$a u$};
    \node (z') at (0.5,0) {$z'$};
    \node (zmin) at (0.5  , -1.875) {$z_{min}$};
    \node (a) at (-2,0.5) {$a$};
    \node (zu') at (-1.5,-0.5) {$ z' u^{-1}$};
    \node (a0) at (-2,-2.5) {$a_{\min}$};

    \draw[->]  (au) -- (y');
    \draw[->] (x) -- (au) ;
    \draw[->] (a) -- (au);
    \draw[->] (zu') -- (z');
    \draw[->] (z') -- (au);
    \draw[->] (zmin) -- (z') ;
    \draw[->] (a0) -- (zmin);
    \draw[->] (a0) -- (zu');
    \draw[->] (zu') -- (a);
    \draw[->] (x) -- (z');
    \draw[->] (x) -- (zmin);
\end{tikzpicture}
\caption{}

\end{figure}

\subsection{The general case}

Let $v \in W_0$. Let 
$$\Adm(\mu)_{\le v}  =  \{w\in \tW\mid w\le t^{v'(\mu)} \text{ for some }v'\le v\}$$
and let $\widehat{\Adm(\mu)_{\le v}} = \{\hat{1}\}\sqcup \Adm(\mu)_{\le v}$. Let $\hat{\eta}_{\le v}$ be the restriction of $\hat{\eta}$ on $\widehat{\Adm(\mu)_{\le v}}$. Let $w \in \Adm(\mu)_{\le v}$. Define $z_{\min}$ and $a_{\min}$ in the same way as in \S\ref{sec:4.1}. It is easy to see that $a_{\min} \le z_{\min}\le v$. Then the chain \ref{eq:4.1} lies in $\widehat{\Adm(\mu)_{\le v}}$. Hence, \ref{eq:4.1} is the unique label-increasing maximal chain and is lexicographically minimal in $[w,\hat{1}] \cap \widehat{\Adm(\mu)_{\le v}}$. This completes the proof of Theorem \ref{thm:main}.

\subsection{Further questions}
Let $\s$ be a length-preserving automorphism of $\tW$. Let $W_{\aff}\subseteq \tW$ be the affine Weyl group. Let $\t$ be the length-zero element in $\tW$ with $t^{\mu} \in W_{\aff}\t$. For an element $w = w'\t \in  W_{\aff}\t$, define $\supp_{\s}(w)$, the \emph{$\s$-support} of $w$, as the smallest $\Ad(\t)\circ\s$-stable subset of $\D_{\aff}$ containing the support of $w'$. A subset $K\subseteq \D_{\aff}$ is called \emph{spherical} if $W_K$ is finite. Let $K$ be a spherical subset. Let 
$${}^{K}\Adm(\mu)_0 = \{w \in \Adm(\mu) \cap {}^{K}\tW \mid \supp_\s(w)\text{ is spherical}   \}.   $$ An element $c=c'\t\in W_{\aff}\t$ is called a \textit{partial $\Ad(\t) \circ \s$-Coxeter element}, if $c'$ is a product of simple reflections, at most one from each $\Ad(\t)\circ \s$-orbit of $\tilde\BS$. Let $${}^{K}\text{Cox}(\mu) = \{w \in {}^{K}\Adm(\mu)_0 \mid w \text{ is a partial }\Ad(\t)\circ\s\text{-Coxeter element}\}.$$ 
We say $(\tW, \s, \mu, K)$ is of Coxeter type, if ${}^{K}\text{Cox}(\mu) = {}^{K}\Adm(\mu)_0$. Note that the set ${}^K \Adm(\mu)_0$ and ${}^{K}\text{Cox}(\mu)$ are closely related to the basic locus of the local model. The set ${}^K \text{Cox}(\mu)$ is the index set of the EKOR stratification of basic loci of Coxeter type and ${}^K \Adm(\mu)_0$ is the index set of the EKOR stratification of basic loci of fully Hodge-Newton decomposable type (cf. \cite{GH10}, \cite{GHN2} and \cite{GHN3}). 

It is natural to ask whether these subsets are dual shellable. These questions are closely related to whether certain basic loci are Cohen-Macaulay. We propose the following conjecture.

\begin{conjecture} Suppose $(\tW,\s,\mu,K)$ is of Coxeter type, then the set $\widehat{{}^{K}\text{Cox}(\mu)} = {}^{K}\text{Cox}(\mu) \sqcup \{\hat{1}\}$ is dual EL-shellable.\footnote{The partial order in ${}^{K}\text{Cox}(\mu)$ is the refined order $\le_{K,\s}$. However, by \cite[Proposition 2.8]{GHN3}, in the Coxeter type case, this refined order is equivalent to the usual Bruhat order.}
\end{conjecture}

In the Coxeter type case $(A_4, \s = \id, \mu = \omega_1^{\vee}+\omega_{4}^{\vee} , K=\BS_0)$, the set ${}^{K}\text{Cox}(\mu)$ is as in Figure 3. It is easy to construct the desired labelling on this set.

\begin{figure}[h]
\centering

\centering
\begin{tikzpicture}
    \node (empty) at (0,0) {$1$};
    \node (0) at (0,1) {$s_0$};
    \node (04) at (-1,2) {$s_{0}s_{4}$};
    \node (01) at (1,2) {$s_{0}s_{1}$};
    \node (043) at (-2 , 3) {$s_0s_4s_3$};
    \node (041) at (0, 3) {$s_0s_4s_1$};
    \node (012) at (2,3) {$s_0s_1s_2$};
    \node (0432) at (-3,4) {$s_0s_4s_3s_2$};
    \node (0431) at (-1,4) {$s_0s_4s_3s_1$};
    \node (0412) at (1,4) {$s_0s_4s_1s_2$};
    \node (0123) at (3,4) {$s_0s_1s_2s_3$};

    \draw[-]  (empty) -- (0);
    \draw[-]  (0) -- (04);
    \draw[-]  (0) -- (01);
    \draw[-]  (04) -- (043);
    \draw[-]  (04) -- (041);
    \draw[-]  (01) -- (041);
    \draw[-]  (01) -- (012);
    \draw[-]  (043) -- (0432);
    \draw[-]  (043) -- (0431);
    \draw[-]  (041) -- (0431);
    \draw[-]  (041) -- (0412);
    \draw[-]  (012) -- (0412);
    \draw[-]  (012) -- (0123);

\end{tikzpicture}
\caption{}

\end{figure}

It is also worth noting that ${}^{K}\Adm(\mu)_0$ is, in general, not dual shellable. For example, in the case $(A_2, \s =\id , \mu =\omega_1^{\vee}+\omega_2^{\vee} =(1,0,-1) , K =\emptyset)$, it is easy to see that the set ${}^{K}\Adm(\mu)_0$ is not a $3$-Cohen-Macaulay poset. See Figure 4. 

\begin{figure}[h]
\centering

\centering
\begin{tikzpicture}
    \node (empty) at (0,0) {$1$};
    \node (0) at (-1,1) {$s_0$};
    \node (1) at (0,1) {$s_1$};
    \node (2) at (1,1) {$s_2$};
    \node (20) at (-2.5, 2) {$s_{2}s_0$};
    \node (02) at (-1.5, 2) {$s_{0}s_2$};
    \node (01) at (-0.5,2) {$s_0s_1$};
    \node (10) at (0.5,2) {$s_1s_0$};
    \node (21) at (1.5,2) {$s_2s_1$};
    \node (12) at (2.5,2) {$s_1s_2$};
    \node (020) at (-2,3) {$s_0s_2s_0$};
    \node (010) at (0,3) {$s_0s_1s_0$};
    \node (121) at (2,3) {$s_1s_2s_1$};
    \draw[-]  (empty) -- (0);
    \draw[-]  (empty) -- (1);
    \draw[-]  (empty) -- (2);
    \draw[-]  (0) -- (20);
    \draw[-]  (0) -- (02);
    \draw[-]  (0) -- (01);
    \draw[-]  (0) -- (10);
    \draw[-]  (1) -- (01);
    \draw[-]  (1) -- (10);
    \draw[-]  (1) -- (21);
    \draw[-]  (1) -- (12);
    \draw[-]  (2) -- (21);
    \draw[-]  (2) -- (12);
    \draw[-]  (2) -- (20);
    \draw[-]  (2) -- (02);
    \draw[-]  (20) -- (020);
    \draw[-]  (02) -- (020);
    \draw[-]  (01) -- (010);
    \draw[-]  (10) -- (010);
    \draw[-]  (21) -- (121);
    \draw[-]  (12) -- (121);
\end{tikzpicture}
\caption{}
\end{figure}


\section{Cohen-Macaulayness of Local Models}

\subsection{Notation}\label{sec:notation}
We adopt the conventions of \cite[§4.1--4.5]{PRS} and \cite[\S 8.1]{HH17}. Let $F$ be a nonarchimedean local field and $L$ the completion of the maximal unramified extension in a fixed separable closure $F^{\mathrm{sep}}/F$. Let $\G_0 \coloneqq \Gal(L^{\mathrm{sep}}/L)$.  

For a connected reductive $L$-group $G$, let $S \subset G$ be a maximal $L$-split torus with centralizer $T$---a maximal torus by Steinberg's theorem. Denote the absolute and relative Weyl groups by $W \coloneqq N_G(T)(L^{\mathrm{sep}})/T(L^{\mathrm{sep}})$ and $W_0 \coloneqq N_G(T)(L)/T(L)$, respectively. 

The reduced root datum $\Sigma \coloneqq (X^*, X_*, \Phi, \Phi^\vee)$ of $(G,T)$ yields an affine Weyl group $W_{\mathrm{aff}}(\Sigma)$ isomorphic to the Iwahori-Weyl group $\widetilde{W}(G_{\mathrm{sc}})$ of the simply-connected cover $G_{\mathrm{sc}}$. Let $V \coloneqq X_*(T)_{\G_0} \otimes \mathbb{R} \cong X_*(S) \otimes \mathbb{R}$ denote the apartment for $S$, equipped with a fixed special vertex and the alcove $\bar{\mathbf{a}}$ in the antidominant chamber whose closure contains this vertex.

The Iwahori-Weyl group $\widetilde{W}(G)$ decomposes as: $X_*(T)_{\G_0} \rtimes W_0$ and $W_{\mathrm{aff}}(\Sigma) \rtimes \Omega$, where $\Omega \subset \widetilde{W}(G)$ stabilizes $\bar{\mathbf{a}}$. The torsion subgroup $X_*(T)_{\G_0,\mathrm{tors}} \subset X_*(T)_{\G_0}$ is central in $\widetilde{W}(G)$, as $W_0$ acts trivially on it.

Let $(-)^\flat$ denote the quotient by $X_*(T)_{I,\mathrm{tors}}$. This induces isomorphisms:
\[
\widetilde{W}(G)^\flat \cong X_*(T)^\flat_{\G_0} \rtimes W_0 \cong W_{\mathrm{aff}}(\Sigma) \rtimes \Omega^\flat,
\]
where $\Sigma$ is reinterpreted as a root datum with $X_* = X_*(T)^\flat_{\G_0}$. For $x \in \widetilde{W}(G)$, write $x^\flat$ for its image in $\widetilde{W}(G)^\flat$.

For a $G$-conjugacy class $\{\mu\} \subset X_*(G) \cong X_*(T)/W$, define $\widetilde{\Lambda}_{\{\mu\}}$ to be $B$-dominant representatives in $\{\mu\}$ for some $L$-rational Borel $B \supset T$ and $\Lambda_{\{\mu\}}$ to be the image of $\widetilde{\Lambda}_{\{\mu\}}$ in $X_*(T)_{\G_0}$. Define $$\mathrm{Adm}(\{\mu\}) \coloneqq \{ x \in \widetilde{W}(G) \mid x \leq t^\lambda \ \text{for some} \ \lambda \in \Lambda_{\{\mu\}} \}$$
with Bruhat order $\leq$ from the decomposition $\widetilde{W}(G) = W_{\mathrm{aff}}(\Sigma) \rtimes \Omega$. 

The following result is proved in \cite[\S 8.2]{HH17}.
\begin{lemma}\label{lem:flat_adm}
\begin{enumerate}
    \item[(i)] For $\tau \in \Omega$ and $x,y \in W_{\mathrm{aff}}(\Sigma)\tau$: $x \leq y \iff x^\flat \leq y^\flat$.
    \item[(ii)] If $\Lambda_{\{\mu\}} \subset W_{\mathrm{aff}}(\Sigma)\tau$, then $\mathrm{Adm}(\{\mu\})$ is the preimage of $\mathrm{Adm}(\Lambda^\flat_{\{\mu\}})$ under $W_{\mathrm{aff}}(\Sigma)\tau \to \widetilde{W}(G)^\flat$, where $\Lambda^\flat_{\{\mu\}} \subset X_*(T)^\flat_{\G_0}$.
\end{enumerate}
\end{lemma}

\subsection{Schubert varieties} Let $w \in \tW$ and $S_w$ be the associated Schubert variety in the affine flag variety. In the mixed characteristic case or in the equal characteristic case with $p \nmid \pi_1(G_{\text{der}})$, the Schubert variety $S_w$ is normal and Cohen-Macaulay. It is discovered in \cite[Theorem 1.1]{HLR} that in the equal characteristic case with $p \mid \pi_1(G_{\text{der}})$, only finitely many Schubert varieties are normal, and any non-normal Schubert variety is not Cohen-Macaulay. In this case, one needs to pass to the seminormalization. Following the stack project, the seminormalization $\tilde S_w$ is the initial scheme mapping universally homeomorphically to $S_w$ with the same residue field. The following result is due to Fakhruddin, Haines, Louren\c co, and Richarz in \cite[Theorem 4.1]{FHLR}. 

\begin{theorem}\label{thm:CMforS}
    The seminormalized Schubert variety $\tilde S_w$ is normal and Cohen-Macaulay. 
\end{theorem}

\subsection{Local models}
In this subsection, we consider the local models with Iwahori level structure. Let $\bG$ be a connected reductive group over $F$ and $\{\mu\}$ be a (not necessarily minuscule) conjugacy class of geometric cocharacters defined over the reflex field $E$, a finite separable extension of $F$. 


A uniform definition of these local models is not available; their construction has developed through several key works. The foundational group-theoretic construction for tamely ramified groups was established by Zhu \cite{Zhu14} and Pappas–Zhu \cite{PZ13}. This was later extended beyond the tamely ramified case by Levin \cite{Le16}, Lourenço \cite{Lo19+}, and Fakhruddin–Haines–Lourenço–Richarz \cite{FHLR}. In equal characteristic, a construction for arbitrary groups was given by Richarz \cite{Ri16}, while the mixed-characteristic case has been largely addressed by the aforementioned authors. We note that in mixed characteristic, constructions depend on the choice of a parahoric group scheme lifting. Finally for minuscule $\mu$, the work of Ansch\"utz-Gleason-Lourenço-Richarz \cite{AGLR} provides a unique projective flat weakly normal scheme representing Scholze's diamond local model.

For our purposes, the essential feature is that all these constructions, when available, yield a local model satisfying the following property:

(*) The local $\tilde M_{\{\mu\}}$ attached to $(\bG, \{\mu\})$ is a flat scheme, whose generic fiber is the seminormalized Schubert variety $\tilde S_{\bG, \{\mu\}}$ attached to $\{\mu\}$, and the reduced special fiber is equal to $\cup_{w \in \Adm(\{\mu'\})} \tilde S_w'$. 

Here $(\bG', \{\mu'\})$ is the equicharacteristic analogues of $(\bG, \{\mu\})$ (see \cite[\S 2]{FHLR}), $\breve I'$ is the standard Iwahori subgroup and $\cup_{w \in \Adm(\{\mu'\})} \tilde S_w'$ is the $\mu'$-admissible locus in the equicharacteristic partial affine flag variety. 

\begin{theorem}\label{thm:CM}
    Let $\tilde M_{\{\mu\}}$ be a local model for $(\bG, \{\mu\})$ satisfying the property $(*)$. Then $\tilde M_{\{\mu\}}$ is Cohen-Macaulay. 
\end{theorem}

\begin{proof}
We consider the seminormalized Schubert variety $\tilde S'_w$ in the affine flag variety of $\bG'$. By Theorem \ref{thm:CMforS}, for any $w$, $\tilde S'_w$ is Cohen-Macaulay. Moreover, the closure of $\tilde S'_w$ equals $\cup_{w' \le w} \tilde S'_{w'}$. Let $X=\cup_{w \in \Adm(\{\mu'\})} \tilde S'_{w}$. By Corollary \ref{cor:main}, $\Adm(\{\mu'\})$ is $N$-Cohen-Macaulay. Hence by \cite[Proposition 4.24]{Go01}, $X$ is Cohen-Macaulay.  

Note that the generic fiber of $\tilde M_{\{\mu\}}$ is the seminormalization of a single Schubert variety, and hence is Cohen-Macaulay by Theorem \ref{thm:CMforS}. By \cite[Lemma 5.7]{HR23}, the whole local model $\tilde M_{\{\mu\}}$ is Cohen-Macaulay. 
\end{proof}

Combining Corollary \ref{cor:CM} with \cite[Lemma 4.22]{Go01}, we have the following result. 

\begin{proposition}
    Let $v \in W_0$. Then $\cup_{v' \le v} \tilde S'_{t^{v'(\mu)}}$ is Cohen-Macaulay. 
\end{proposition}

\printbibliography

\end{document}